\documentclass[12pt,reqno]{amsart}

\usepackage{amsmath}
\usepackage{amssymb}
\usepackage{amsthm}
\usepackage{amscd}
\usepackage{xypic}
\usepackage{hyperref}

\headsep=1cm \numberwithin{equation}{section}
\def\P{{\mathbb P}}

\newtheorem{theorem}{Theorem}[section]
\newtheorem{lemma}[theorem]{Lemma}

\newtheorem{proposition}[theorem]{Proposition}
\newtheorem{corollary}[theorem]{Corollary}

\theoremstyle{definition}
\newtheorem{definition}[theorem]{Definition}
\newtheorem{remark}[theorem]{Remark}

\newtheorem{notation and conventions}[theorem]{Notation and Conventions}
\newtheorem{convention and reminder}[theorem]{Convention and Reminder}
\newtheorem{convention and remark}[theorem]{Convention and Remark}
\newtheorem{definition and remark}[theorem]{Definition and Remark}

\newtheorem{reminders and definition}[theorem]{Reminders and Definition}

\newtheorem{remark and notations}[theorem]{Remark and Notations}
\newtheorem{notation and remark}[theorem]{Notation and Remark}
\newtheorem{example}[theorem]{Example}

\newcommand\reg{\operatorname{reg}}
\newcommand\Ker{\operatorname{\Ker}}

\makeatletter
\newcommand{\stars}{}
\DeclareRobustCommand{\stars}[1]{\stars@{#1}}
\newcommand{\stars@}[1]{%
  \ifcase#1\relax\or\stars@one\or\stars@two\or\stars@three\or\stars@four
  \else ??\fi
}
\newcommand{\stars@char}{$\scriptstyle*$}
\newcommand{\stars@base}[1]{%
  $\m@th\vcenter{\offinterlineskip\ialign{\hfil##\hfil\cr#1\crcr}}$%
}
\newcommand{\stars@one}{%
  \stars@base{\stars@char}%
}
\newcommand{\stars@two}{%
  \stars@base{\stars@char\cr\stars@char}%
}
\newcommand{\stars@three}{%
  \stars@base{\stars@char\cr\stars@char\stars@char}%
}
\newcommand{\stars@four}{%
  \stars@base{\stars@char\stars@char\cr\stars@char\stars@char}%
}
\makeatother

\title{Castelnuovo-Mumford regularity of finite schemes}

\begin{document}

\author{Donghyeop Lee}
\address{Department of Mathematics, Korea University, Seoul 136-701, Korea}
\email{porket333@korea.ac.kr}

\author{Euisung Park}
\address{Department of Mathematics, Korea University, Seoul 136-701, Korea}
\email{euisungpark@korea.ac.kr}

\date{Seoul, \today}


\maketitle \thispagestyle{empty}

\begin{abstract}
Let $\Gamma \subset \P^n$ be a nondegenerate finite subscheme of degree $d$. Then the Castelnuovo-Mumford regularity ${\rm reg} ({\Gamma})$ of $\Gamma$ is at most $\left\lceil \frac{d-n-1}{t(\Gamma)} \right\rceil +2$ where $t(\Gamma)$ is the smallest integer such that $\Gamma$ admits a $(t+2)$-secant $t$-plane. In this paper, we show that ${\rm reg} ({\Gamma})$ is close to this upper bound if and only if there exists a unique rational normal curve $C$ of degree $t(\Gamma)$ such that $\reg (\Gamma \cap C) = \reg (\Gamma)$.
\end{abstract}

\maketitle  \setcounter{page}{1}

\section{Introduction}
\noindent Throughout this paper, we work over an algebraically closed field $\mathbb{K}$ of characteristic $0$. A closed subscheme $\Gamma \subset \P^n$, defined by a sheaf of ideals $\mathcal{I}_{\Gamma}$ in $\mathcal{O}_{\P^n}$, is said to be $p$-regular in the sense of Castelnuovo-Mumford if
\begin{equation*}
H^{i}(\mathbb{P}^n,\mathcal{I}_{\Gamma}(p-i))=0 \ \text{ for all } \ i\geq{1}.
\end{equation*}
The interest in this concept stems partly from the fact that $\Gamma$ is $p$-regular if and only if for every $j \geq 0$ the minimal generators of the $j$-th syzygy module of the saturated homogeneous ideal $I (\Gamma)$ of $\Gamma$ in the homogeneous coordinate ring $R := \mathbb{K} [x_0 , x_1 , \ldots, x_n ]$ of $\P^n$ occur in degree $\leq p+j$ (cf. \cite{EG}). In particular, $I (\Gamma)$ can be generated by forms of degree $\leq p$. The Castelnuovo-Mumford regularity ${\rm reg} ({\Gamma})$ of ${\Gamma}$ is defined to be the least integer $p$ such that ${\Gamma}$ is $p$-regular. In this paper, we study the regularity of finite subschemes of $\mathbb{P}^n$.

Let $\Gamma \subset \P^n$ be a nondegenerate finite subscheme of degree $d \geq n+3$. There is a well-known upper bound of ${\rm reg}(\Gamma)$ in terms of some basic invariants of $\Gamma$. To state precisely, recall that a subspace $\Lambda$ of $\P^n$ is said to be a $(t+2)$-secant $t$-plane to $\Gamma$ if it is a $t$-dimensional subspace of $\P^n$ such that
\begin{equation*}
{\rm length} (\mathcal{O}_{\mathbb{P}^n} / (\mathcal{I}_{\Lambda} +\mathcal{I}_{{\Gamma}})) \geq t+2
\end{equation*}
where $\mathcal{I}_{\Lambda}$ is the sheaf of ideals of $\Gamma$ in $\P^n$. We denote by $t(\Gamma)$ the smallest integer $t$ such that $\Gamma$ admits a $(t+2)$-secant $t$-plane. Then $1 \leq t(\Gamma) \leq n$. Also, $\Gamma$ is said to be \textit{in linearly general position} if $t(\Gamma)$ is equal to $n$. It always holds that
\begin{equation}\label{ineq:kwak}
{\rm reg}(\Gamma) \leq \left\lceil \frac{d-n-1}{t(\Gamma)} \right\rceil +2
\end{equation}
where $\lceil{a}\rceil$ is the smallest integer greater than or equal to $a$ (cf. \cite[Proposition 2.1]{K}). Briefly speaking, this shows that in order for ${\rm reg}(\Gamma)$ to become large, the value of $t(\Gamma)$ must be small. Along this line, our paper aims to explore answers to the following problem.\\

\begin{enumerate}
  \item[$(*)$] When ${\rm reg}(\Gamma)$ is close to the upper bound $\left\lceil \frac{d-n-1}{t(\Gamma)} \right\rceil +2$ in (\ref{ineq:kwak}), find some elementary geometric reasons for why $\Gamma$ is not $({\rm reg}(\Gamma)-1)$-regular.\\
\end{enumerate}

\noindent In the case of $t(\Gamma)=1$, the answer to this problem was obtained in \cite{LPW}, where it is proved that if
\begin{equation*}
\left\lceil \frac{d-n-1}{2} \right\rceil +3 \leq {\rm reg}(\Gamma) \leq  d-n+1
\end{equation*}
(and hence $t(\Gamma) =1$), then $\Gamma$ admits a ${\rm reg}(\Gamma)$-secant line. This result provides an intuitive and geometric explanation for why $\Gamma$ fails to be $({\rm reg}(\Gamma)-1)$-regular. 

Regarding the problem $(*)$ and the above result in \cite{LPW}, our first main result in this paper is as follows :

\begin{theorem}\label{thm:main1}
Let $\Gamma \subset \P^n$ be a nondegenerate finite subscheme of degree $d$ such that 
\begin{equation}\label{eq:1}
\left\lceil \frac{d-n-1}{t(\Gamma)+1} \right\rceil +3 \leq {\rm reg}(\Gamma) \leq \left\lceil \frac{d-n-1}{t(\Gamma)} \right\rceil +2.
\end{equation}
Then there exists a unique $t(\Gamma)$-dimensional subspace $\Lambda$ of $\P^n$ such that
\begin{enumerate}
\item[$(i)$] $\Gamma \cap \Lambda$ is nondegenerate and in linearly general position as a subscheme of $\Lambda$
\end{enumerate}
and
\begin{enumerate} 
\item[$(ii)$] ${\rm reg}(\Gamma \cap \Lambda) = {\rm reg}(\Gamma)$. 
\end{enumerate}
\end{theorem}
  
The proof of Theorem \ref{thm:main1} is in Section 3.

Note that if ${\rm reg}(\Gamma)$ holds the inequalities in (\ref{eq:1}), then $d$ is at least $n+ t(\Gamma)+2$.

Theorem \ref{thm:main1} generalizes the previously mentioned result in \cite{LPW}. That is, if
\begin{equation*}
\left\lceil \frac{d-n-1}{2} \right\rceil +3 \leq {\rm reg}(\Gamma) \leq  d-n+1
\end{equation*}
(and hence $t(\Gamma) =1$), then Theorem \ref{thm:main1} says that there exists a unique line $\Lambda = \P^1$ such that ${\rm reg}(\Gamma \cap \Lambda) ={\rm reg}(\Gamma)$. That is, $\Gamma$ admits a unique ${\rm reg}(\Gamma)$-secant line.\\

Theorem \ref{thm:main1} gives us a partial answer for the problem $(*)$. That is, if ${\rm reg}(\Gamma)$ satisfies the inequality (\ref{eq:1}) then there exists a subscheme $\Gamma \cap \Lambda$ of $\Gamma$ which is in linearly general position in $\Lambda$ and which satisfies ${\rm reg}(\Gamma \cap \Lambda) = {\rm reg}(\Gamma)$. This leads us to study finite schemes in linearly general position whose regularity is close to the upper bound in (\ref{ineq:kwak}). We begin with recalling the following definition.

\begin{definition}
A nondegenerate finite set $\Gamma \subset \P^n$ of $d$ points is said to be \textit{in uniform position} if 
\begin{equation*}
h_{Y} (\ell) = \min \{ |Y| , h_{\Gamma} (\ell) \}
\end{equation*}
for any subset $Y$ of $\Gamma$ and any $\ell \geq 0$, where $h_{\Gamma} (t)$ and $h_{Y} (t)$ are the Hilbert functions of $\Gamma$ and $Y$, respectively. 
\end{definition}

Any finite set in uniform position is always in linearly general position. When $\Gamma$ is in uniform position and $d$ is large enough, the maximal regularity case was classified as follows.

\begin{theorem}[N. V. Trung and G. Valla in \cite{TV} and U. Nagel in \cite{N}]\label{thm:Nagel}
Let ${\Gamma} \subset \mathbb{P}^n$ be a finite set of $d$ points in uniform position. Then

\renewcommand{\descriptionlabel}[1]%
{\hspace{\labelsep}\textrm{#1}}
\begin{description}
\setlength{\labelwidth}{13mm} \setlength{\labelsep}{1.5mm}
\setlength{\itemindent}{0mm}

\item[{\rm (1)}] Suppose that $d>(n+1)^2$. Then
$${\rm{reg}}(\Gamma)=\left\lceil\frac{d-1}{n}\right\rceil+1 \text{ if and only if }{\Gamma} \text{ lies on a rational normal curve}.$$

\item[{\rm (2)}] Suppose that $n\geq{3}$, $d\geq min(3n+1,2n+5)$ and ${\Gamma}$ does not lie on a rational normal curve. Then ${\rm reg}({\Gamma}) \leq \tau (n,d)$ where
$$\tau (n,d) := \begin{cases} \left\lceil\frac{d}{n+1}\right\rceil+2
\quad\text{ if } n+1 \text{ divides } d, \text{and}\\
\left\lceil\frac{d}{n+1}\right\rceil+1 \quad\text{ otherwise}. \end{cases}$$
\end{description}
\end{theorem}

Theorem \ref{thm:Nagel}.$(1)$ was shown using the classical Castelnuovo Lemma.

Theorem \ref{thm:Nagel} says that if ${\Gamma}$ is in uniform position and $d>(n+1)^2$, then either $\Gamma$ lies on a rational normal curve and so ${\rm{reg}}(\Gamma)$ is maximal or else ${\rm{reg}}(\Gamma)$ is at most $\tau (n,d)$.\\

From now on, let ${\Gamma} \subset \mathbb{P}^n$ be a finite scheme in linearly general position of degree $d \geq n+3$. We first define the integer ${\rho}({\Gamma})$ as
\begin{equation*}
\rho (\Gamma)={\rm max} \{\vert {\Gamma} \cap C \vert ~|~ C \text{ is a rational normal curve in } \P^n \}.
\end{equation*} 
By \cite[Theorem 1]{EH}, any finite scheme of degree $n+3$ in linearly general position is contained in a unique rational normal curve. Therefore, it always holds that
$$n+3\leq {\rho}({\Gamma})\leq d.$$
Also, Theorem \ref{thm:Nagel}.$(1)$ says that if ${\Gamma}$ is in uniform position and $d \geq (n+1)^2$, then ${\rm{reg}}(\Gamma)$ is maximal if and only if $\rho (\Gamma)=d$. 

Along this line, our second main result in the present paper is as follows.

\begin{theorem}\label{thm:main2}
Let ${\Gamma} \subset \P^n$, $n \geq 2$, be a finite subscheme of degree $d$ in linearly general position such that
\begin{equation}\label{eq:assumption on reg}
d \geq 4n^2+6n+1 \quad \mbox{and} \quad {\rm reg}({\Gamma}) \geq \left\lceil\frac{d-1}{n+\frac{n}{2n+2}}\right\rceil+3.
\end{equation}
Then
\begin{equation}\label{eq:intersection with RNC}
\rho(\Gamma)> d-(m+1)n \quad \mbox{where} \quad  m := \left\lceil\frac{d-1}{n}\right\rceil+1 - {\rm reg}({\Gamma}).
\end{equation}
Furthermore, there is a unique rational normal curve $C$ of degree $n$ such that
$$\rho(\Gamma)=\vert\Gamma\cap{C}\vert \quad\mbox{and} \quad {\rm reg}(\Gamma) = {\rm reg}(\Gamma \cap C).$$
\end{theorem}

The proof of Theorem \ref{thm:main2} is in Section 4.

Considering the inequalities in (\ref{ineq:kwak}) and (\ref{eq:assumption on reg}), the following condition is required:
\begin{equation*}
\left\lceil\frac{d-1}{n+\frac{n}{2n+2}}\right\rceil+3 \leq \left\lceil\frac{d-1}{n}\right\rceil+1  
\end{equation*} 
Note that our assumption $d\geq4n^2+6n+1$ ensures that this inequality always holds.

Contrary to the results of Theorem \ref{thm:Nagel}, it can happen that if $\Gamma \subset \P^n$ is in linearly general position but not in uniform position then   
\begin{equation*}
\left\lceil\frac{d}{n+1}\right\rceil+2 < {\rm{reg}}(\Gamma) \leq \left\lceil\frac{d-1}{n}\right\rceil+1.
\end{equation*}
Theorem \ref{thm:main2} provides a very precise description of such finite subschemes. In particular, it says that if ${\rm{reg}}(\Gamma)$ is close to the maximal possible value then there exists a unique rational normal curve which contains most of $\Gamma$. Thus Theorem \ref{thm:main2}  
clearly shows what properties hold when the condition on $\Gamma$ is extended from the ``uniform position" to ``linearly general position".

Next, we apply Theorem \ref{thm:main2} to finite subschemes in linearly general position having maximal regularity. This can be compared to Theorem \ref{thm:Nagel}.$(1)$.

\begin{corollary}\label{cor:maximal case}
Let ${\Gamma} \subset \P^n$, $n \geq 2$, be a finite subscheme of degree $d \geq 4n^2+6n+1$ in linearly general position. Write $d = nq+r+2$ for some $0 \leq r \leq n-1$.
Then
\begin{equation*}
{\rm reg}({\Gamma}) = \left\lceil\frac{d-1}{n}\right\rceil+1 \ \text{  if and only if  } \
{\rho}({\Gamma})\geq d-r.
\end{equation*}
In particular, if $r=0$ and $\reg (\Gamma) = \left\lceil\frac{d-1}{n}\right\rceil+1$ then $\Gamma$ lies in a rational normal curve and hence it is in uniform position.
\end{corollary}

The proof of Corollary \ref{cor:maximal case} is in Section 5.\\  

We finish this section by providing an answer to the previous problem $(*)$ by combining Theorem \ref{thm:main1} and Theorem \ref{thm:main2}.
 
\begin{theorem}\label{thm:main3}
Let $\Gamma \subset \P^n$ be a nondegenerate finite subscheme of degree $d$ with $t(\Gamma) = t$. If
\begin{equation*}
\left \lceil{\frac{d-1}{t+\frac{t}{2t +2} }}\right \rceil+2 < {\rm{reg}}(\Gamma) \leq \left\lceil \frac{d-n-1}{t} \right\rceil +2,
\end{equation*}
then there exist a unique subspace $\P^{t}$ of $\P^n$ and a unique rational normal curve $C$ in $\P^{t}$ such that ${\rm{reg}}(\Gamma\cap C)={\rm{reg}}(\Gamma)$.
\end{theorem}
 
The proof of Theorem \ref{thm:main3} is in Section 4.

Regarding Problem $(*)$, if $\Gamma \subset \P^n$ is as in Theorem \ref{thm:main3} then it fails to be $(\rm{reg}(\Gamma)-1)$-regular since there exists a rational normal curve $C$ in $\P^{t}$ such that ${\rm{reg}}(\Gamma\cap C)={\rm{reg}}(\Gamma)$. \\
 
\noindent {\bf Organization of the paper.} In $\S 2$, we review some basic facts to study the Castelnuovo-Mumford regularity of finite schemes. In $\S 3$ and $\S 4$, we give a proof of Theorem \ref{thm:main1} and a proof of Theorem \ref{thm:main2}, respectively. Finally, in $\S 5$ we prove more properties of finite schemes in linearly general position which have maximal regularity.\\
 
\noindent {\bf Acknowledgement.} This work was supported by the National Research Foundation of Korea(NRF) grant funded by the Korea government(MSIT) (No. 2022R1A2C1002784).

\section{Preliminaries}
\noindent We fix a few notations, which we use throughout this paper.

\begin{notation and conventions}\label{notrmk:basicfacts}
Let ${\Gamma} \subset \P^n$ be a (possibly degenerate) finite subscheme, defined by a sheaf of ideals $\mathcal{I}_{\Gamma}$ in $\mathcal{O}_{\P^n}$. Also, let
$$I (\Gamma) := \bigoplus_{\ell \in \mathbb{Z}} H^0 (\P^n ,\mathcal{I}_{\Gamma} (\ell))$$
be the saturated homogeneous ideal of $\Gamma$ in the homogeneous coordinate ring $R := \mathbb{K} [x_0 , x_1 , \ldots, x_n ]$ of $\P^n$.

\renewcommand{\descriptionlabel}[1]%
{\hspace{\labelsep}\textrm{#1}}
\begin{description}
\setlength{\labelwidth}{13mm} \setlength{\labelsep}{1.5mm}
\setlength{\itemindent}{0mm}

\item[$(1)$] The length of $\Gamma$, defined to be the length of $\mathcal{O}_\Gamma = \mathcal{O}_{\P^n} /\mathcal{I}_{\Gamma}$, is equal to $h^{0}(\Gamma ,\mathcal{O}_\Gamma)$. It is also called the degree of $\Gamma$. We will denote it by $\vert \Gamma \vert$.

\item[$(2)$] For each $m \in \mathbb{Z}$, let
\[
\rho_m : H^0(\mathbb{P}^n , \mathcal{O}_{\mathbb{P}^n} (m))\rightarrow H^0({\Gamma} , \mathcal{O}_{{\Gamma}} (m))
\]
 be the natural restriction map. The Hilbert function $h_{\Gamma}$ of ${\Gamma}$ is defined by
$$h_{\Gamma} (m)= \dim_{\mathbb{K}} {\rm Im} ( \rho_m ).$$
We say that $\Gamma$ is $m$-\textit{normal} if $\rho_m$ is surjective, or equivalently, if $\Gamma$ is $(m+1)$-regular.

\item[$(3)$] Let $V$ be a hypersurface of $\P^n$. The intersection ${\Gamma}\cap{V}$ is defined to be the scheme defined by the ideal sheaf $(\mathcal I_{{\Gamma}}+\mathcal I_V )$.
Also, \textit{the residual scheme} of ${\Gamma}$ with respect to $V$, denoted by ${\Gamma} : V$, means the scheme defined by the ideal sheaf $(\mathcal I_{{\Gamma}} : \mathcal I_V )$.
\item[$(4)$] If $\dim\langle\Gamma\rangle\geq{1}$, we define $t(\Gamma)$ to be the largest integer $k$ such that $\dim \langle \Gamma'\rangle=\vert\Gamma'\vert-1$
 for any subscheme $\Gamma'$ of $\Gamma$ of degree $\vert\Gamma'\vert\leq k+1$.
It is elementary to see that $1\leq t(\Gamma)\leq\dim \langle \Gamma\rangle$.
For example, if $\vert\Gamma\vert\geq 3$, then the statement that $t(\Gamma)=1$ is equivalent to the statement that $\Gamma$ has a trisecant line.
\end{description}
\end{notation and conventions}

Next, we list a few well-known facts about finite subschemes in a projective space below.

\begin{proposition}\label{Basic facts}
Let ${\Gamma} \subset \P^n$ be a finite subscheme of degree $d$. Then

\renewcommand{\descriptionlabel}[1]%
{\hspace{\labelsep}\textrm{#1}}
\begin{description}
\setlength{\labelwidth}{13mm} \setlength{\labelsep}{1.5mm}
\setlength{\itemindent}{0mm}
\item[$(1)$] If ${\Gamma}$ is $m$-normal, then any subscheme ${\Gamma}'$ of ${\Gamma}$ is also $m$-normal.

\item[$(2)$] If ${\Gamma}'$ is a subscheme of ${\Gamma}$ such that $\dim~\langle{{\Gamma}'} \rangle = n'$ and $\vert{\Gamma}'\vert=d'$,
then
$$\dim ~\langle{{\Gamma}}\rangle \leq \min \{ n'+d-d' , n\}.$$

\item[$(3)$] Let ${\Gamma}_0$ be a subscheme of ${\Gamma}$ such that $\vert{\Gamma}_0\vert=d_0$. Then there exist subschemes ${\Gamma}_1$, ${\Gamma}_2$,$\ldots$, ${\Gamma}_{d-d_0-1}$ of $\Gamma$ such that
\begin{equation*}
{\Gamma}_0 \subsetneq {\Gamma}_1 \subsetneq \cdots \subsetneq {\Gamma}_{d-d_0-1} \subsetneq {\Gamma} \quad \mbox{and} \quad |{\Gamma}_i| = |{\Gamma}_0| + i.
\end{equation*}

\item[$(4)$] Let $V$ be a hypersurface of $\mathbb{P}^n$. Then
$$|{\Gamma}| = |{\Gamma} \cap V| + |{\Gamma}:V|.$$

\item[$(5)$] $($La méthode d’Horace, \cite{Hi}$)$ Let $V$ be a hypersurface of degree $k$. If ${\Gamma} \cap V$ is $m$-normal and ${\Gamma}:V$ is $(m-k)$-normal, then ${\Gamma}$ is $m$-normal.

\item[$(6)$] If $\Gamma$ is in linearly general position and lies on a rational normal curve of degree $n$, Then
\[
{\rm reg}({\Gamma}) =\left\lceil\frac{d-1}{n}\right\rceil+1.
\]
\end{description}
\end{proposition}

\begin{proof}
For $(1) \sim (3)$, we refer the reader to \cite[Proposition 2.2]{LPW}. The proof of $(4)$ comes immediately from the following exact sequence
\begin{equation*}
0 \rightarrow \big( S/(I_{{\Gamma}}:V) \big) (k) \rightarrow S/I_{{\Gamma}} \rightarrow S/ \langle I_{{\Gamma}},V \rangle \rightarrow 0
\end{equation*}
where $k$ is the degree of $V$. For $(5)$, see \cite[p.352]{Hi}.
For $(6)$, see \cite[Proposition 3.4]{P} (cf. \cite[Proposition 2.2]{N}, \cite[Proposition 2.3]{NP}).
\end{proof}

\begin{remark}\label{remark:Horace}
In this paper, we will apply the above-mentioned d'Horace Lemma as follows.
Let $\Gamma$ and $V$ be as in Proposition \ref{Basic facts}.$(5)$ such that ${\rm reg}({\Gamma}) > {\rm reg}({\Gamma\cap{V}})$. Then
$${\rm reg}({\Gamma}:V ) \geq {\rm reg}(\Gamma )-k.$$
Indeed, if not then $\Gamma \cap V$ is $({\rm reg}({\Gamma})-2)$-normal and $\Gamma:V$ is (${\rm reg}({\Gamma})-k-2)$-normal. Thus it follows by Proposition \ref{Basic facts}.$(5)$ that $\Gamma$ is (${\rm reg}({\Gamma})-2$)-normal, which is obviously a contradiction.
\end{remark}

\section{Proof of Theorem \ref{thm:main1}}
\noindent This section is devoted to giving a proof of Theorem \ref{thm:main1}.
We begin with a definition.

\begin{definition}\label{def:RPH condition}
We say that a nondegenerate finite subscheme $\Gamma \subset \P^n$ satisfies condition \textmd{RPH} (=regularity preserving hyperplane) if it admits a regularity preserving hyperplane section in the sense that there exists a hyperplane $H = \P^{n-1}$ such that the following two statements hold:
\begin{enumerate}
\item[$(i)$] $\Gamma\cap{H}$ is a nondegenerate subscheme of $H$;
\item[$(ii)$] the equality ${\rm reg}(\Gamma \cap H) = {\rm reg}(\Gamma)$ holds.
\end{enumerate}
\end{definition}

\begin{remark} \label{rm:RPH condition2}
$(1)$ Let $H$ be a hyperplane such that ${\rm reg}(\Gamma\cap{H})={\rm reg}(\Gamma)$ but $\Gamma\cap{H}$ fails to span $H$.
 In such a case, we can choose another hyperplane $H'$ containing $\Gamma \cap {H}$ and such that $\Gamma\cap{H'}$ is a nondegenerate subscheme of $H'$.
Then
\[
{\rm reg}(\Gamma)={\rm reg}(\Gamma\cap{H})\leq {\rm reg}(\Gamma\cap{H}')\leq {\rm reg}(\Gamma).
\]
That is, condition \textmd{RPH} holds for $\Gamma$ if there is a hyperplane which satisfies only condition $(ii)$ in Definition \ref{def:RPH condition}.

\noindent $(2)$ For example, if a finite subscheme $\Gamma \subset \P^n$ is in linearly general position and $|\Gamma| \geq n+2$, then ${\rm reg}(\Gamma)\geq 3$ and any nondegenerate hyperplane section of $\Gamma$ is $2$-regular.
Thus, $\Gamma$ fails to satisfy condition \textmd{RPH}.
\end{remark}

Our first goal in this section is to verify that $\Gamma \subset \P^n$ in Theorem \ref{thm:main1} satisfies condition \textmd{RPH}.

\begin{lemma}\label{lem:not preserving hyperplane section}
Let $\Gamma \subset \P^n$ be as in Theorem \ref{thm:main1} and let $H = \P^{n-1}$ be a hyperplane such that
\begin{equation*}
|\Gamma \cap H | \geq n+1 \quad \mbox{and} \quad {\rm reg}(\Gamma \cap H ) < {\rm reg}(\Gamma).
\end{equation*}
Then the following statements hold.
\smallskip

\renewcommand{\descriptionlabel}[1]%
             {\hspace{\labelsep}\textrm{#1}}
\begin{description}
\setlength{\labelwidth}{13mm}
\setlength{\labelsep}{1.5mm}
\setlength{\itemindent}{0mm}

\item[$(1)$] ${\rm reg}(\Gamma ) \geq 4$.
\item[$(2)$] ${\rm reg}(\Gamma:H) \geq {\rm reg}(\Gamma) -1$.
\item[$(3)$] $|\Gamma:H| - \dim  \langle\Gamma:H \rangle + t(\Gamma:H) \leq d-n-1$.
\item[$(4)$] $t(\Gamma:H) = t(\Gamma)$.
\end{description}
\end{lemma}

\begin{proof}
$(1)$ Since $d \geq n+t(\Gamma)+2$, we have ${\rm reg}(\Gamma) \geq 4$  from the first inequalities in (\ref{eq:1}).
\noindent $(2)$ See Remark \ref{remark:Horace}.

\noindent $(3)$ It always holds that $t(\Gamma:H) \leq \dim \langle \Gamma:H \rangle$. This shows that
\begin{align*}
n-\dim \langle\Gamma:H \rangle + t(\Gamma:H) &< n+1 \\
&\leq |\Gamma \cap H| \\
&= d- |\Gamma:H|.
\end{align*}
Thus we get the desired inequality $|\Gamma:H| - \dim  \langle\Gamma:H \rangle + t(\Gamma:H) \leq d-n-1$.

\noindent $(4)$ First, note that $\dim\langle{\Gamma:H}\rangle\geq 1$ since ${\rm reg}(\Gamma:H)>1$.  By inequality (\ref{ineq:kwak}) and Lemma \ref{lem:not preserving hyperplane section}.$(3)$, we have
\begin{align*}
{\rm reg}(\Gamma:H) &\leq \left\lceil \frac{|\Gamma:H|-\dim \langle \Gamma:H \rangle -1 }{t(\Gamma:H)} \right\rceil +2 \\
&\leq \left\lceil \frac{d-n-1-t(\Gamma:H)}{t(\Gamma:H)} \right\rceil +2 = \left\lceil \frac{d-n-1}{t(\Gamma:H)} \right\rceil+1.
\end{align*}
Then, we can combine this result with inequalities (\ref{eq:1}) and Lemma \ref{lem:not preserving hyperplane section}.$(2)$ to deduce that
\begin{equation*}
\left\lceil \frac{d-n-1}{t(\Gamma)+1} \right\rceil +2  \leq {\rm reg}(\Gamma)-1 \leq {\rm reg}(\Gamma:H) \leq \left\lceil \frac{d-n-1}{t(\Gamma:H)} \right\rceil+1 .
\end{equation*}
Here, the inequality
\begin{equation*}
\left\lceil \frac{d-n-1}{t(\Gamma)+1} \right\rceil +2  \leq \left\lceil \frac{d-n-1}{t(\Gamma:H)} \right\rceil+1
\end{equation*}
implies that $t(\Gamma:H) \leq t(\Gamma)$.
To see this, suppose that $t(\Gamma:H) < t(\Gamma)$ for the sake of contradiction.
Since $\dim\langle{\Gamma:H}\rangle\geq 1$, we have $t(\Gamma:H) \leq \dim \langle \Gamma:H \rangle$ by the definition of $t(\Gamma:H)$.
Remark that
\[
|\Gamma:H|> \dim\langle\Gamma:H\rangle+1
\]
since ${\rm reg}(\Gamma:H)\geq 3$.
If $t(\Gamma:H) < \dim \langle \Gamma:H \rangle$, then $\Gamma:H$ admits a $t(\Gamma:H)$-dimensional subspace $\Lambda$ such that $\vert\Lambda\cap{(\Gamma:H)}\vert\geq t(\Gamma:H)+2$ hence so does $\Gamma$.
This is impossible since $t(\Gamma:H)<t(\Gamma)$.
Thus, we may assume that $t(\Gamma:H) = \dim \langle \Gamma:H \rangle$.
Then, it immediately implies that $\dim \langle \Gamma:H \rangle$ is strictly less than $t(\Gamma)$.
If $\vert \Gamma:H \vert \geq \dim \langle \Gamma:H \rangle+2$, then $t(\Gamma)\leq t(\Gamma:H)$.
 This implies that $\vert \Gamma:H \vert = \dim \langle \Gamma:H \rangle+1$ and hence ${\rm reg}(\Gamma:H)=2$.
This is a contradiction since
\[
{\rm reg}(\Gamma:H) \geq {\rm reg}(\Gamma)-1 \geq 3
\]
 by Lemma \ref{lem:not preserving hyperplane section}.$(1)$ and $(2)$.
Therefore, we can deduce that $t(\Gamma:H) = t(\Gamma)$.
\end{proof}

\begin{theorem}\label{thm:Thm RPH}
Let $\Gamma \subset \P^n$ be a nondegenerate finite subscheme of degree $d$. If $t(\Gamma) \leq n-1$ and
\begin{equation}\label{eq:Thm RPH}
\left\lceil \frac{d-n-1}{t(\Gamma)+1} \right\rceil +3 \leq {\rm reg}(\Gamma) \leq \left\lceil \frac{d-n-1}{t(\Gamma)} \right\rceil +2,
\end{equation}
then $\Gamma \subset \P^n$ satisfies condition \rm\textmd{RPH}.
\end{theorem}

\begin{proof}
It suffices to show that there exists a hyperplane $H$ such that ${\rm reg}(\Gamma\cap{H})={\rm reg}(\Gamma)$ by $(1)$ in Remark \ref{rm:RPH condition2}. Since ${\rm reg}(\Gamma)$ satisfies the inequalities in (\ref{eq:Thm RPH}) we have $d \geq n+ t(\Gamma)+2$.
Keeping this in mind, we will use induction on $d$.

First, consider the initial case that $d=n+t(\Gamma)+2$. We can choose a hyperplane $H$ such that $\vert\Gamma \cap H\vert \geq n+1$ since $\Gamma$ is not in linearly general position(i.e., $t(\Gamma)\leq n-1$). Now, it is easy to see that $\vert\Gamma:H\vert \leq d-n-1 = t(\Gamma)+1$ and hence 
$$\dim \langle \Gamma:H \rangle \leq |\Gamma:H|-1 \leq t(\Gamma).$$
The inequality $\dim \langle \Gamma:H \rangle \leq t(\Gamma)$ implies that $\dim \langle \Gamma:H \rangle = |\Gamma:H|-1$ by the definition of $t(\Gamma)$. Thus, we conclude that $1\leq {\rm reg}(\Gamma:H)\leq 2$. On the other hand, it follows from inequalities in (\ref{eq:Thm RPH}) that ${\rm reg}(\Gamma)=4$. Therefore, we deduce that ${\rm reg}(\Gamma \cap H ) =4$ by the Horace method (cf. Proposition \ref{Basic facts}.(5)).

Next, we assume that $d > n+t(\Gamma)+2$. In order to derive a contradiction, we suppose that there is no hyperplane $H$ such that ${\rm reg}(\Gamma\cap{H})={\rm reg}(\Gamma)$. As before, we can choose a hyperplane $H_0$ such that $\vert\Gamma \cap H_0\vert \geq n+1$ since $\Gamma$ is not in linearly general position. We have  ${\rm reg}(\Gamma)-1\leq {\rm reg}(\Gamma:H_0)$ by the Horace method.

We claim that there is a hyperplane $H' = \P^{n-1}$ such that ${\rm reg}(\Gamma \cap H' ) = {\rm reg}(\Gamma) -1$. To see this, consider the case that $\dim \langle \Gamma:H_0 \rangle < n$. From this, we can choose a hyperplane $H'$ containing a subscheme $\Gamma:H_0$.
Then,
\begin{equation*}
{\rm reg}(\Gamma)-1 \leq {\rm reg}(\Gamma:H_0) \leq {\rm reg}(\Gamma \cap H' ) < {\rm reg}(\Gamma)
\end{equation*}
and hence we get the desired equality
\begin{equation*}
{\rm reg}(\Gamma \cap H' ) = {\rm reg}(\Gamma)-1.
\end{equation*}
Next, consider the other case that $\dim \langle \Gamma:H_0 \rangle = n$. From Lemma \ref{lem:not preserving hyperplane section}.$(3)$, we have
\begin{equation*}
|\Gamma:H_0| - \dim  \langle \Gamma:H_0 \rangle + t(\Gamma:H_0) \leq d-n-1.
\end{equation*}
Remark that $t(\Gamma:H_0) = t(\Gamma)$ by Lemma \ref{lem:not preserving hyperplane section}.$(4)$.
Then we obtain
\begin{equation*}
\frac{|\Gamma:H_0| - \dim  \langle \Gamma:H_0 \rangle -1}{t(\Gamma)+1} +1  \leq  \frac{d-n-1}{t(\Gamma)+1}.
\end{equation*}
Thus it holds that
\begin{equation*}
\left \lceil \frac{|\Gamma:H_0| - \dim  \langle \Gamma:H_0 \rangle -1}{t(\Gamma)+1} \right\rceil +2  \leq  \left \lceil \frac{d-n-1}{t(\Gamma)+1} \right\rceil +1\leq {\rm reg}(\Gamma)-2 \leq {\rm reg}(\Gamma:H_0)-1.
\end{equation*}
In particular, it is shown that
\begin{equation*}
\left \lceil \frac{|\Gamma:H_0| - \dim  \langle \Gamma:H_0 \rangle -1}{t(\Gamma)+1} \right\rceil +3  \leq  {\rm reg}(\Gamma:H_0).
\end{equation*}
By induction hypothesis, there is a hyperplane $H' = \P^{n-1}$ such that
\[
{\rm reg}((\Gamma:H_0) \cap H' ) = {\rm reg}(\Gamma:H_0).
\]
It follows that
\begin{equation*}
{\rm reg}(\Gamma)-1 \leq {\rm reg}(\Gamma:H_0)= {\rm reg}( (\Gamma:H_0)\cap H') \leq {\rm reg}(\Gamma \cap H' ) \leq {\rm reg}(\Gamma)-1
\end{equation*}
and hence we get the desired equality ${\rm reg}(\Gamma \cap H' ) = {\rm reg}(\Gamma)-1$. Moreover, we may assume that $\Gamma\cap{H'}$ is nondegenerate in $H'$(cf. see in Remark \ref{rm:RPH condition2}).

For this hyperplane $H'$, we obtain the following inequalities :
\begin{equation*}
\left \lceil{\frac{d-n-1}{t(\Gamma)+1}}\right \rceil+2\leq {\rm reg}(\Gamma)-1 \leq {\rm reg}(\Gamma:{H'}) \leq \left \lceil{\frac{\vert{\Gamma:{H'}}\vert-\dim\langle \Gamma:{H'}\rangle-1}{t(\Gamma:{H'})}}\right \rceil+2
\end{equation*}
(cf. inequalities (\ref{eq:Thm RPH}) and Lemma \ref{lem:not preserving hyperplane section}.$(2)$).
From the inequality \begin{equation*}
\left \lceil{\frac{d-n-1}{t(\Gamma)+1}}\right \rceil+2\leq  \left \lceil{\frac{\vert{\Gamma:{H'}}\vert-\dim\langle \Gamma:{H'}\rangle-1}{t(\Gamma:{H'})}}\right \rceil+2,
\end{equation*} we can derive the inequality
\begin{equation*}
\frac{d-n-1}{t(\Gamma)+1} < \frac{|\Gamma:{H'}|- \dim \langle \Gamma:{H'} \rangle -1}{t(\Gamma)}+1
\end{equation*}
since $t(\Gamma)=t(\Gamma:{H'})$ by Lemma \ref{lem:not preserving hyperplane section}.$(4)$.
Since $t(\Gamma) = t(\Gamma:{H'}) \leq \dim \langle \Gamma:{H'} \rangle$ and $|\Gamma:{H'}| = d-|\Gamma\cap{H'}|$ , we can derive the follwoing inequalities :
\begin{equation*}
|\Gamma\cap{H'}| < n-\dim \langle \Gamma:{H'} \rangle  +t(\Gamma)+\frac{d-n-1}{t(\Gamma)+1} \leq n+ \left\lceil \frac{d-n-1}{t(\Gamma)+1} \right\rceil.
\end{equation*}
Note that ${\rm reg}(\Gamma\cap{H'}) \leq |\Gamma\cap{H'}|-(n-1)+1$ from the inequality (\ref{ineq:kwak}).

Thus we can deduce that
\begin{equation*}
{\rm reg}(\Gamma) = {\rm reg}(\Gamma\cap{H'}) +1 \leq |\Gamma\cap{H'}|-(n-1)+2 < 3+\left \lceil{\frac{d-n-1}{t(\Gamma)+1}}\right \rceil,
\end{equation*}
which contradicts our assumption (\ref{eq:Thm RPH}).
This shows that $\Gamma \subset \P^n$ satisfies condition \textmd{RPH}.
\end{proof}

$\newline$
\noindent {\bf Proof of Theorem \ref{thm:main1}.} By Theorem \ref{thm:Thm RPH}, there is a hyperplane $H$ such that $\Gamma\cap{H}$ is a nondegenerate subscheme of $H$ and ${\rm reg}(\Gamma\cap{H})={\rm reg}(\Gamma)$.
Then, we have
\begin{align*}
\left \lceil{\frac{\vert\Gamma\cap{H}\vert-\text{dim}\langle\Gamma\cap{H}\rangle-1}{t(\Gamma\cap{H})+1}}\right \rceil+2
&\leq\left \lceil{\frac{d-n-1}{t(\Gamma)+1}}\right \rceil+2 \\
&<{\rm reg}(\Gamma)\\
&={\rm reg}(\Gamma\cap{H}) \\
&\leq\left \lceil{\frac{\vert\Gamma\cap{H}\vert-\text{dim}\langle\Gamma\cap{H}\rangle-1}{t(\Gamma\cap{H})}}\right \rceil+2.
\end{align*}
Since $\vert\Gamma\cap{H}\vert-\text{dim}\langle\Gamma\cap{H}\rangle-1\leq d-n-1$, the inequality
\[
\left \lceil{\frac{d-n-1}{t(\Gamma)+1}}\right \rceil+2
<\left \lceil{\frac{\vert\Gamma\cap{H}\vert-\text{dim}\langle\Gamma\cap{H}\rangle-1}{t(\Gamma\cap{H})}}\right \rceil+2
\]implies that ${t(\Gamma\cap{H})}\leq t(\Gamma)$.
As in the proof of Lemma \ref{lem:not preserving hyperplane section}.(4), we see that $t(\Gamma\cap{H})=t(\Gamma)$.
If $t(\Gamma\cap{H})<n-1$, then we can use Theorem \ref{thm:Thm RPH} again for $\Gamma\cap{H}$.
Therefore, we can obtain the desired existence by using Theorem \ref{thm:Thm RPH} $n-t(\Gamma)$ times.

To prove the uniqueness, we suppose that there are two distinct $t(\Gamma)$-dimensional linear spaces $\Lambda_1$ and $\Lambda_2$ satisfying conditions (i),(ii) in Theorem\ref{thm:main1}.
Let $d_1$ and $d_2$ denote $\vert\Gamma\cap\Lambda_1\vert$ and  $\vert\Gamma\cap\Lambda_2\vert$ respectively.
Then
\begin{equation}
\left \lceil{\frac{d-n-1}{t(\Gamma)+1}}\right \rceil+3\leq {\rm reg}(\Gamma)= {\rm reg}(\Gamma\cap\Lambda_i) \leq \left \lceil{\frac{d_i-1}{t(\Gamma)}}\right \rceil+2
\end{equation}
for $i=1,2$.
From this, we get the inequalities $\frac{d-n-1}{t(\Gamma)+1}+1<\frac{d_i-1}{t(\Gamma)}$ for $i=1,2$.
These imply that
\begin{equation}\label{ieq:uniqueness}
t(\Gamma)\left(\frac{d-n-1}{t(\Gamma)+1}\right)+1<d_i
\end{equation} for $i=1,2$.
Let $r$ denote the dimension of $\Lambda_1\cap\Lambda_2$, with the convention that $r=-1$ when $\Lambda_1\cap\Lambda_2=\emptyset$.
Thus, the dimension of the linear space spanned by $\Lambda_1\cup\Lambda_2$ is $2t-r$.
Since $\Gamma\cap\Lambda_1$ is in linearly general position in $\Lambda_1$, the degree of $\Gamma\cap\Lambda_1\cap\Lambda_2$ is less than or equal to $r+1$.
We can obtain an inequality
\begin{equation*}
d_1+d_2-(r+1)\leq\vert(\Gamma\cap\Lambda_1)\cup(\Gamma\cap\Lambda_2)\vert\leq\vert\Gamma\cap\langle\Lambda_1\cup\Lambda_2\rangle\vert.
\end{equation*}
Note that $\vert\Gamma\vert-\vert\Gamma\cap(\Lambda_1\cup\Lambda_2)\vert\geq n-\dim\langle\Lambda_1\cup\Lambda_2\rangle$ since $\Gamma$ is nondegenerate in $\mathbb{P}^n$.
Then, we can deduce that
\[
d-(d_1+d_2-(r+1))\geq n-(2t(\Gamma)-r).
\]
Combined with (\ref{ieq:uniqueness}), we can derive an inequality
\[
2t(\Gamma)\left(\frac{d-n+t(\Gamma)}{t(\Gamma)+1}\right)+1+n-2t(\Gamma)<d.
\]
Then we have $2t(\Gamma)(d-n+t(\Gamma))+(1+n-2t(\Gamma))(t(\Gamma)+1)<d(t(\Gamma)+1)$, and hence
\[
0<(d-n-1)(1-t(\Gamma)).
\]
 Since $d>n+1$, we have $t(\Gamma)<1$.
Obviously, this is a contradiction. This completes the proof.
\qed

\section{Proof of Theorem \ref{thm:main2}}
\noindent This section is primarily devoted to giving a proof of Theorem \ref{thm:main2}, but also includes a proof of Theorem \ref{thm:main3}.
We begin with the following interesting observation.

\begin{lemma}\label{lem:lemma1}
Let ${\Gamma} \subset \mathbb{P}^n$ be a finite subscheme of degree $d$ in linearly general position and let $m$ be a nonnegative integer such that
$$d \geq (m+3)n \quad \mbox{and} \quad {\rm reg}({\Gamma}) = \left\lceil\frac{d-1}{n}\right\rceil-{m}+1.$$
Let $Q \subset \mathbb{P}^n$ be a quadric such that $\vert{\Gamma}\cap{Q}\vert\geq ({m}+3)n$.
Then
$${\rm reg}({\Gamma}\cap{Q})={\rm reg}({\Gamma}) \quad \mbox{and} \quad \vert{\Gamma}\cap{Q}\vert>d-({m}+1)n.$$
\end{lemma}

\begin{proof}
If $\Gamma \subset Q$, then we are done. Suppose that $\Gamma \nsubseteq Q$, and hence $\vert{\Gamma}:{Q}\vert >0$. Note that ${\Gamma}\cap{Q}$ and ${\Gamma}:Q$ are in linearly general position. Also,
$$d=|\Gamma| = |\Gamma \cap Q| + |\Gamma :Q | \geq (m+3)n+|\Gamma:Q|$$
and hence
$${\rm reg}({\Gamma}) = \left\lceil\frac{d-1}{n}\right\rceil-{m}+1 \geq  \left\lceil\frac{|\Gamma:Q|-1}{n}\right\rceil+4.$$
Concerning ${\Gamma}:Q$, all cases can be divided into the following three cases:\\
\begin{enumerate}
\item[$(i)$] $\vert{\Gamma}:{Q}\vert = 1$; or
\item[$(ii)$] $2 \leq \vert{\Gamma}:{Q}\vert \leq n$; or
\item[$(iii)$] $n+1\leq\vert{\Gamma}:{Q}\vert$.\\
\end{enumerate}

\noindent In case $(i)$, we have ${\rm reg}({\Gamma}:{Q})=1$ and ${\rm reg}({\Gamma}) \geq{4}$. In case $(ii)$, we have ${\rm reg}({\Gamma}:{Q})=2$ and ${\rm reg}({\Gamma}) \geq 5$. In case $(iii)$, it holds that
$${\rm reg}(\Gamma : Q ) \leq  \left\lceil\frac{d-|\Gamma \cap Q|-1}{n}\right\rceil+1  \leq  \left\lceil\frac{d-(m+3)n-1}{n}\right\rceil+1 = {\rm reg}(\Gamma)-3. $$
Thus, by Proposition \ref{Basic facts}.$(5)$ and Remark \ref{remark:Horace}, it must be true that
$${\rm reg}({\Gamma}\cap{Q})={\rm reg}({\Gamma}).$$
To prove the remaining part, let us suppose on the contrary that
$$\vert{\Gamma}\cap{Q}\vert\leq d-({m}+1)n.$$
Then, we have
\begin{equation*}
{\rm reg}({\Gamma})= {\rm reg}({\Gamma}\cap{Q}) \leq \left\lceil \frac{d-({m}+1)n-1}{n}\right\rceil+1 = {\rm reg}({\Gamma})-1.
\end{equation*}
This is a contradiction, and so the desired inequality $\vert{\Gamma}\cap{Q}\vert>d-({m}+1)n$ holds.
\end{proof}

\begin{lemma}\label{lem:4.2}
Let ${\Gamma} \subset \P^n$ be a nondegenerate finite subscheme of degree $d$ in linearly general position such that
$$\left\lceil\frac{d-1}{n+\frac{n}{2n+2}}\right\rceil+3\leq {\rm reg}({\Gamma}).$$
Let $Q \subset \P^n$ be a quadric such that $\vert Q\cap{\Gamma}\vert\geq 2n+1$ and ${\rm reg}({\Gamma}\cap{Q})<{\rm reg}({\Gamma})$. Then
$$\left\lceil\frac{\vert{\Gamma}:Q\vert-1}{n+\frac{n}{2n+2}}\right\rceil+3\leq {\rm reg}({\Gamma}:Q).$$
Also, let $m$ and $m'$ be nonnegative integers such that
$${\rm reg}({\Gamma})=\left\lceil\frac{d-1}{n}\right\rceil-{m}+1 \quad \text{and} \quad {\rm reg}({\Gamma}:Q)=\left\lceil\frac{\vert{\Gamma}:{Q}\vert-1}{n}\right\rceil-{m'}+1.$$
Then $m \geq m'$.
\end{lemma}

\begin{proof}
Since ${\rm reg}({\Gamma}\cap{Q})<{\rm reg}({\Gamma})$, Proposition \ref{Basic facts}.$(5)$ implies that ${\rm reg}({\Gamma}:Q)\geq{\rm reg}({\Gamma})-2$.
 Also, it follows from the assumption $\vert{\Gamma}\cap{Q}\vert\geq 2n+1$ that $\vert{\Gamma}:Q\vert\leq d-2n-1$.
Hence we have the following :
\begin{align*}
{\rm reg}({\Gamma}:Q) & \geq {\rm reg}({\Gamma})-2 \\
& \geq \left\lceil\frac{d-1}{n+\frac{n}{2n+2}}\right\rceil+1 = \left\lceil\frac{d-2(n+\frac{n}{2n+2})-1}{n+\frac{n}{2n+2}}\right\rceil+3 \\
& \geq  \left\lceil\frac{d-2n-1-1}{n+\frac{n}{2n+2}}\right\rceil+3 \\
& \geq \left\lceil\frac{\vert{\Gamma}:Q\vert-1}{n+\frac{n}{2n+2}}\right\rceil+3.
\end{align*}
Next, we will verify the second statement.
We first note that
$$\left\lceil\frac{\vert{\Gamma}:Q\vert}{n}\right\rceil-{m}+1\leq\left\lceil\frac{d-1}{n}\right\rceil-{m}-1 ={\rm reg}({\Gamma})-2 $$
since $\vert{\Gamma}:Q\vert\leq d-2n-1$.
Now, it follows from the inequality ${\rm reg}({\Gamma})-2\leq {\rm reg}({\Gamma}:Q)$ that
$$\left\lceil\frac{\vert{\Gamma}:Q\vert}{n}\right\rceil-{m}+1 \leq {\rm reg}({\Gamma}:Q)=\left\lceil\frac{\vert{\Gamma}:{Q}\vert-1}{n}\right\rceil-{m'}+1.$$
Obviously, this implies the desired inequality ${m} \geq {m}'$.
\end{proof}

Now, we are ready to give a proof of Theorem \ref{thm:main2}.\\

\noindent {\bf Proof of Theorem \ref{thm:main2}.}
First, note that we can write
$${\rm reg}({\Gamma})=\left\lceil\frac{d-1}{n}\right\rceil-{m}+1 \mbox{ for a nonnegative integer} \ m$$
since ${\rm reg}({\Gamma})\leq\left\lceil\frac{d-1}{n}\right\rceil+1$.
Then, it follows from
$$\left\lceil\frac{d-1}{n+\frac{n}{2n+2}}\right\rceil+3\leq {\rm reg}({\Gamma})=\left\lceil\frac{d-1}{n}\right\rceil-{m}+1$$
 that
\begin{equation*}
\frac{d-1}{n+\frac{n}{2n+2}}+3<\frac{d-1}{n}-{m}+2.
\end{equation*}
Thus, we can deduce that
\begin{equation}\label{ineq:3.1}
(1+{m})(2n^2+3n)+1<d.
\end{equation}

Next, we will show that \\
\begin{enumerate}
\item[$(**)$] there is a unique rational normal curve $C$ of degree $n$ such that ${\rm reg}(\Gamma) = {\rm reg}(\Gamma \cap C)$.\\
\end{enumerate}

\noindent Before proving the above statement $(**)$, note that $(**)$ implies the two desired properties
$$\rho(\Gamma)> d-(m+1)n \quad \mbox{and} \quad  \rho(\Gamma)=\vert\Gamma\cap{C}\vert.$$
Indeed, we have
\[
\left\lceil\frac{d-1}{n}\right\rceil-m+1={\rm reg}({\Gamma})={\rm reg}(\Gamma\cap{C})\leq\left\lceil\frac{\vert \Gamma\cap{C}\vert-1}{n}\right\rceil+1.
\]
Hence the first statement, $\rho(\Gamma)> d-(m+1)n$, comes immediately from
\[
\frac{d-1}{n}-m+1\leq\frac{\vert \Gamma\cap{C}\vert-1}{n}+2.
\]
For the second statement, $\rho(\Gamma)=\vert\Gamma\cap{C}\vert$, assume that there is a rational normal curve $C'$ of degree $n$ such that $\vert\Gamma\cap{C'}\vert>\vert\Gamma\cap{C}\vert$. Then, we have
\[
{\rm reg}({\Gamma}) = {\rm reg}(\Gamma\cap{C}) = \left\lceil\frac{\vert \Gamma\cap{C}\vert-1}{n}\right\rceil+1 \leq \left\lceil\frac{\vert \Gamma\cap{C'}\vert-1}{n}\right\rceil+1 = {\rm reg}(\Gamma\cap{C'}) \leq {\rm reg}(\Gamma).
\]
Thus, it follows immediately that
$${\rm reg}(\Gamma\cap{C'})={\rm reg}(\Gamma).$$
This contradicts the uniqueness of $C$.
As a result, it follows that $\rho(\Gamma)=\vert\Gamma\cap{C}\vert$.
Therefore, it suffices to prove $(**)$. From now on, we will prove $(**)$ by showing the existence and the uniqueness of the rational normal curve $C$. \\

\noindent \underbar{\textmd{Existence :}} To prove the existence part in $(**)$, we will proceed in three steps.\\

\textmd{Step 1.}\quad We will construct a subscheme ${\Gamma}'$ of ${\Gamma}$ with $\vert{\Gamma}'\vert  \geq 2{m}{n}+4n+3$ such that \\
\begin{enumerate}
\item[]{(\stars{3})}
${\rm reg}({\Gamma}'\cap{Q})={\rm reg}({\Gamma}')$ for any quadric $Q \subset \mathbb{P}^n$ with $\vert Q\cap{\Gamma}'\vert \geq 2n+1$.\\
\end{enumerate}
If ${\Gamma}$ satisfies the property $($\stars{3}$)$, then we are done. Suppose that ${\Gamma}$ fails to satisfy $($\stars{3}$)$. That is, there exists a quadric $Q_0$ such that $$\vert{\Gamma}\cap{Q_0}\vert \geq 2n+1 \quad \mbox{and} \quad {\rm reg}({\Gamma}) > {\rm reg}({\Gamma}\cap{Q_0}).$$
Let ${\Gamma}_1$ denote subscheme ${\Gamma}:{Q_0}$. If $\Gamma_1$ fails to satisfy $($\stars{3}$)$, then there exists a quadric $Q_1$ such that
$$\vert{\Gamma_1}\cap{Q_1} \vert \geq 2n+1 \quad \mbox{and} \quad {\rm reg}({\Gamma_1}) > {\rm reg}({\Gamma_1}\cap{Q_1}).$$
In this way, setting ${\Gamma}={\Gamma}_0$, we can obtain a subscheme ${\Gamma}_i$ of ${\Gamma}$ and a quadric $Q_{i-1}$ inductively.
Indeed,
if ${\Gamma}_{i-1}$ fails to satisfy $($\stars{3}$)$ then there exists a quadric $Q_{i-1}$ such that
$$\vert{\Gamma}_{i-1}\cap{Q_{i-1}}\vert \geq 2n+1 \quad \mbox{and} \quad {\rm reg}({\Gamma}_{i-1}) > {\rm reg}({\Gamma}_{i-1}\cap{Q_{i-1}}).$$
Then we define ${\Gamma}_i$ as the subscheme ${\Gamma}_{i-1}:Q_{i-1}$ of $\Gamma_{i-1}$. From our construction, it holds that
$$\vert{\Gamma}_{i}\vert\leq{d-(2n+1)i}.$$
Also, it holds by Proposition \ref{Basic facts}.$(5)$ that
$${\rm reg}({\Gamma}_i) \geq {\rm reg}({\Gamma}_{i-1})-2.$$
It follows that ${\rm reg}({\Gamma}_i) \geq {\rm reg}({\Gamma})-2i$, and hence we obtain
\begin{align*}
\left\lceil\frac{d-1}{n}\right\rceil-{m}+1-2i & \quad = \quad {\rm reg}({\Gamma})-2{i} \\
                                              & \quad \leq \quad {\rm reg}({\Gamma}_i) \\
                                              & \quad \leq \quad \left\lceil\frac{|\Gamma_i|-1}{n}\right\rceil+1 \\
                                              & \quad \leq \quad \left\lceil\frac{d-(2n+1)i-1}{n}\right\rceil+1 \\
                                              & \quad = \quad \left\lceil\frac{d-i-1}{n}\right\rceil-2i+1.
\end{align*}
The inequality
$$\left\lceil\frac{d-1}{n}\right\rceil-{m}-2i+1\leq\left\lceil\frac{d-i-1}{n}\right\rceil-2i+1$$
implies that $i<({m}+1)n$. That is, there exists an integer $i<({m}+1)n$ such that the subscheme ${\Gamma}_i$ of ${\Gamma}$ satisfies the property $($\stars{3}$)$. With this $i$ fixed, we need to show that
\begin{equation}\label{ineq:Gamma i}
\vert{\Gamma}_i\vert  \geq 2mn+4n+3.
\end{equation}
To show this, we first claim that
\begin{equation}\label{ineq:4.2}
\vert{\Gamma}_i\vert  \geq d-(2n+1)(({m}+1)n-1)
\end{equation}
(It is clear that the number $d-(2n+1)(({m}+1)n-1)>0$ by our assumptions).
For the sake of contradiction, suppose that
\begin{equation*}
\vert{\Gamma}_i\vert \leq d-(2n+1)(({m}+1)n-1)-1=d-(2n+1)({m}+1)n+2n.
\end{equation*}
Then we have
\begin{equation}\label{ineq:3.4}
{\rm reg}({\Gamma}_i)\leq\left\lceil\frac{\vert{\Gamma}_i\vert-1}{n}\right\rceil+1\leq\left\lceil\frac{d-1}{n}\right\rceil-m-2n(m+1)+2.
\end{equation}
On the other hand, the inequalities ${\rm reg}({\Gamma}_i) \geq {\rm reg}({\Gamma})-2{i}$ and $i<(m+1)n$ imply that
\begin{equation}\label{ineq:3.5}
{\rm reg}({\Gamma}_i) \geq {\rm reg}({\Gamma})-2i = \left\lceil\frac{d-1}{n}\right\rceil-{m}-2i+1 \geq \left\lceil\frac{d-1}{n}\right\rceil-{m}-2n(m+1)+3.
\end{equation}
It is obvious that (\ref{ineq:3.4}) and (\ref{ineq:3.5}) contradict each other.
Now, it follows from inequalities in (\ref{ineq:3.1}) and (\ref{ineq:4.2}) that
$$\vert{\Gamma}_i\vert \geq (1+{m})(2n^2+3n)+2-(2n+1)(({m}+1)n-1)=2{m}{n}+4n+3.$$
Therefore, this ${\Gamma}_i$ is the set ${\Gamma}'$ we are looking for.\\

\textmd{Step 2.}\quad  We will show that\\

\begin{enumerate}
\item[]{\bf(\stars{4})} ${\rm reg}({\Gamma}\cap{Q})={\rm reg}({\Gamma})$ for any quadric $Q$ in $\mathbb{P}^n$ such that $\vert Q\cap{\Gamma}'\vert\geq 2n+1$.\\
\end{enumerate}

\noindent To prove $($\stars{4}$)$, we will apply Lemma \ref{lem:lemma1} to ${\Gamma}'$ constructed in Step 1.
Since ${\Gamma}'$ is in linearly general position in $\mathbb{P}^n$, we can write
$${\rm reg}({\Gamma}')=\left\lceil\frac{\vert{\Gamma}'\vert-1}{n}\right\rceil-{m'}+1$$
for some nonnegative integer $m'$.
Let $Q$ be a quadric in $\mathbb{P}^n$ such that $\vert{Q\cap{{\Gamma}'}}\vert\geq 2n+1$. Then we have ${\rm reg}({\Gamma}'\cap{Q})={\rm reg}({\Gamma}')$ by $($\stars{3}$)$.
Also, we have
$$\vert{\Gamma}'\cap{Q}\vert>\vert{\Gamma}'\vert-({m}'+1)n$$
by Lemma \ref{lem:lemma1}. Since $\vert{\Gamma}'\vert\geq2{m}{n}+4n+3$, it follows that
$$\vert{\Gamma}'\cap{Q}\vert>2{m}{n}-{m}'{n}+3n+3.$$
Next, we will show that ${m}\geq{m}'$.
For this purpose, we first recall the construction of ${\Gamma}'$ described in Step 1. If we write
$${\rm reg}({\Gamma}_j)=\left\lceil\frac{\vert{\Gamma}_{j}\vert-1}{n}\right\rceil-m_{j}+1$$
for some nonnegative integer $m_j$, then $m_{j+1}\leq m_{j}\leq m$ by Lemma \ref{lem:4.2}.
This shows that ${m}\geq{m}'$. Thus, we have
$$\vert{\Gamma}\cap{Q}\vert\geq\vert{\Gamma}'\cap{Q}\vert>{m}{n}+3n+3.$$
Now, we apply Lemma \ref{lem:lemma1} once again to deduce that ${\rm reg}({\Gamma})={\rm reg}({\Gamma}\cap{Q})$.\\

\textmd{Step 3.}\quad  Setting $X_0={\Gamma}$, we will choose inductively a subscheme $X_i$ of ${\Gamma}$ such that
\begin{equation}\label{ineq:Xi Construction}
\left\lceil\frac{\vert{X_i}\vert-1}{n+\frac{n}{2n+2}}\right\rceil+3\leq {\rm reg}(X_i)={\rm reg}({\Gamma})
\end{equation}
as follows. Let $X_i$ be a subscheme of ${\Gamma}$ satisfying (\ref{ineq:Xi Construction}). If $X_i$ is contained in a rational normal curve, then the proof is completed. Suppose not. Then, we have
$$h_{X_i}(2)>2n+1$$
by \cite[Theorem 3.2 and Theorem 4]{EH}. Also, if we write
$${\rm reg}(X_i) = \left\lceil\frac{\vert{X_i}\vert-1}{n}\right\rceil-m_i+1,$$
then
$$\vert X_i\vert > (1+{m}_{i})(2n^2+3n)+1$$
by (\ref{ineq:3.1}). By Step 1 and Step 2, there exists a subscheme $X_i'$ of $X_i$ such that
$${\rm reg}(X_i\cap{Q})={\rm reg}(X_i) \quad \mbox{for any quadric} \quad Q \subset \mathbb{P}^n \quad \mbox{with} \quad \vert Q\cap{X_i'}\vert \geq 2n+1.$$
Choose a subscheme $A_{i}$ of $X_i'$ with $\vert A_i\vert=2n+1$. Then $h_{A_i}(2)=2n+1$ since $A_i$ is $3$-regular. Thus, we may deduce that there exists a quadric $Q_i$ such that $Q_i$ contains $A_i$ but not $X_i$. We set
$$X_{i+1}:=X_i\cap{Q_i}.$$
Since $2n+1\leq\vert Q_i\cap{X_i'}\vert$, we have
$$\left\lceil\frac{\vert{X}_{i+1}\vert-1}{n+\frac{n}{2n+2}}\right\rceil+3\leq\left\lceil\frac{\vert{X}_i\vert-1}{n+\frac{n}{2n+2}}\right\rceil+3\leq {\rm reg}({X}_{i})={\rm reg}({X}_{i+1}).$$
Note that the sequence $\{\vert X_i\vert\}$ is strictly decreasing for $i$.
However, $\{\vert X_i\vert\}$ has a lower bound since
\[
\left\lceil\frac{d-1}{n+\frac{n}{2n+2}}\right\rceil+3\leq \left\lceil\frac{d-1}{n}\right\rceil-m+1={\rm reg}({\Gamma})={\rm reg}(X_i)\leq\left\lceil\frac{\vert X_i\vert-1}{n}\right\rceil+1.
\]
After a finite number of steps (this number is less than or equal to $mn+n-1$), $X_i$ should lie on a rational normal curve and ${\rm reg}(X_i)={\rm reg}({\Gamma})$.\\

\noindent \underbar{\textmd{Uniqueness :}} Suppose that there are two different rational normal curves $C_1$ and $C_2$ in $\mathbb{P}^n$ such that
\[
{\rm reg}({\Gamma}\cap{C_1})={\rm reg}({\Gamma}\cap{C_2}) = {\rm reg}({\Gamma}).
\]
Then
\[
\left\lceil\frac{\vert{{\Gamma}\cap{C_i}}\vert-1}{n}\right\rceil+1 ={\rm reg}({\Gamma}\cap{C_i}) = {\rm reg}({\Gamma}) \geq \left\lceil\frac{d-1}{n+\frac{n}{2n+2}}\right\rceil+3
\]
for $i=1,2$. Thus, it holds that
$$\frac{\vert{{\Gamma}\cap{C_i}}\vert-1}{n}+2 > \frac{d-1}{n+\frac{n}{2n+2}}+3$$
and hence
$$\vert{{\Gamma}\cap{C_i}}\vert \geq \frac{(2n+2)(d-1)}{2n+3}+n+2.$$
Since the degree of the scheme-theoretic intersection ${{C_1}\cap{C_2}}$ is at most $n+2$ (cf. \cite[Theorem2.1]{EH1}), we have
$$\vert{{\Gamma}\cap{C_1}\cap{C_2}}\vert\leq n+2.$$
Then, we have
\begin{align*}
\vert{ ( \Gamma \cap C_1 ) \cup ( \Gamma \cap C_2 ) } \vert & \quad = \quad \vert{{\Gamma}\cap{C_1}}\vert+\vert{{\Gamma}\cap{C_2}}\vert-\vert{{\Gamma}\cap{C_1}\cap{C_2}}\vert \\
                                                            & \quad \geq \quad \frac{2(2n+2)(d-1)}{2n+3}+n+2.
\end{align*}
Here the last term on the right is strictly bigger than $d$, which is a contradiction. This completes the proof. \qed\\

\noindent {\bf Proof of Theorem \ref{thm:main3}.} If $t(\Gamma)=n$, then we apply Theorem \ref{thm:main2} to $\Gamma$. 

Now, suppose that $t(\Gamma)<n$. Then we first apply Theorem \ref{thm:main1} to $\Gamma$. Thus there exists a unique subspace $\mathbb{P}^{t(\Gamma)}$ of $\mathbb{P}^n$ such that 
$${\rm{reg}}(\Gamma)={\rm{reg}}(\Gamma\cap\mathbb{P}^{t(\Gamma)}).$$ 
Then the proof is completed by applying Theorem \ref{thm:main2} to $\Gamma\cap{\mathbb{P}}^{t(\Gamma)}$. \qed\\

\section{Finite schemes in linearly general position having maximal regularity}
\noindent This section is devoted to studying further properties of finite schemes in linearly general position whose regularities are maximal. 
We begin with proving Corollary \ref{cor:maximal case}.\\

\noindent {\bf Proof of Corollary \ref{cor:maximal case}.}
Suppose that ${\rm{reg}}({\Gamma})=\left\lceil\frac{d-1}{n}\right\rceil+1$. Since $d \geq 4n^2 +6n+1$, it holds that
$$\left\lceil\frac{d-1}{n+\frac{n}{2n+2}}\right\rceil+3 \leq{\rm reg}({\Gamma}) \leq \left\lceil\frac{d-1}{n}\right\rceil+1 $$
and hence Theorem \ref{thm:main2} says that there is a unique rational normal curve $C$ such that $\rho (\Gamma) = \vert\Gamma\cap{C}\vert$ and  ${\rm{reg}}({\Gamma})={\rm{reg}}({\Gamma\cap{C}})$.
Then we have
\[
\left\lceil\frac{d-1}{n}\right\rceil+1={\rm{reg}}({\Gamma})={\rm{reg}}({\Gamma\cap{C}}) = \left\lceil\frac{\vert\Gamma\cap{C}\vert-1}{n}\right\rceil+1.
\]
Since we write $d=nq+r+2$ for $0\leq{r}\leq{n-1}$, it follows that $\rho{(\Gamma)} = \vert\Gamma\cap{C}\vert \geq d-r$.

Conversely, suppose that $\rho(\Gamma)\geq d-r$ and so there exists a rational normal curve $C$ such that $\vert\Gamma\cap{C}\vert \geq d-r$. Since $d=nq+r+2$ for $0\leq{r}\leq{n-1}$, it follows that
$${\rm{reg}}({\Gamma\cap{C}}) = \left\lceil\frac{\vert\Gamma\cap{C}\vert-1}{n}\right\rceil+1 \geq \left\lceil\frac{d-r-1}{n}\right\rceil+1 = \left\lceil\frac{d-1}{n}\right\rceil+1$$
But it also holds that
$${\rm{reg}}({\Gamma\cap{C}})\leq{\rm{reg}}({\Gamma})\leq\left\lceil\frac{d-1}{n}\right\rceil+1.$$
Therefore, it is shown that
$$ {\rm{reg}}({\Gamma}) = {\rm{reg}}({\Gamma\cap{C}}) = \left\lceil\frac{d-1}{n}\right\rceil+1.$$
This completes the proof.     \qed\\
 
Let $\Gamma \subset \P^n$ be a nondegenerate finite subscheme. When we add some points to $\Gamma$, its Castelnuovo-Mumford regularity may either increase or remain the same. In the following theorem, we find a condition such that the latter case occurs. 

\begin{theorem}\label{thm:add several points}
Let $m \geq 1$ be an integer and let $\Gamma \subset \P^n$ be a finite scheme of degree
$$d \geq 4n^2 + 6n+1+ 2(n+1)m$$
which is contained in a rational normal curve $C$ of degree $n$. If $A \subset \P^n \setminus C$ is a finite set such that $|A| \leq m$ and $\Gamma \cup A$ is in linearly general position, then
$$\reg (\Gamma \cup A ) = \reg (\Gamma) = \left\lceil\frac{d-1}{n}\right\rceil+1.$$
\end{theorem}

To give a proof of Theorem \ref{thm:add several points}, we begin with the following lemma. 

\begin{lemma}\label{lem:adding points}
Let ${\Gamma}$ be a nondegenerate finite subscheme of degree $d\geq{4n^2+6n+1}$ in linearly general position in $\mathbb{P}^n$. If there is a rational normal curve $C$ in $\mathbb{P}^n$ such that
$$\frac{2n+2}{2n+3}(d-1)+2n+1\leq\vert{\Gamma\cap{C}}\vert,$$
then ${\rm reg}({\Gamma}\cap{C})={\rm reg}({\Gamma})$.
\end{lemma}

\begin{proof}
Our assumptions imply that
$$\left\lceil\frac{2n+2}{2n^2+3n}(d-1)\right\rceil+3 \leq \left\lceil\frac{\vert{\Gamma\cap{C}}\vert-1}{n}\right\rceil+1 = {\rm reg}({\Gamma}\cap{C})\leq{\rm reg}({\Gamma}).$$
Then we can apply Theorem \ref{thm:main2}, and get a rational normal curve $C'$ such that
$${\rm reg}({\Gamma}\cap{C'})={\rm reg}({\Gamma}).$$
Now, one can show that the two rational normal curves $C$ and $C'$ are equal by using the same idea used in the proof of the uniqueness part of Theorem \ref{thm:main2}.
Thus, we get the desired equality ${\rm reg}({\Gamma}\cap{C})={\rm reg}({\Gamma})$.
\end{proof}

\noindent {\bf Proof of Theorem \ref{thm:add several points}.} By Lemma \ref{lem:adding points}, it suffices to show that
$$\frac{2n+2}{2n+3}(d+|A|-1)+2n+1 \leq d.$$
This inequality is equivalent to
$$4n^2 + 6n+1 +2(n+1)|A|  \leq d,$$
which holds by our assume on $d$ since $|A| \leq m$. This completes the proof.\qed \\

\begin{example}
Let $C \subset \P^5$ be a rational normal curve of degree $5$ and let $\Gamma$ be a subscheme of $C$ of length $10000$. Theorem \ref{thm:add several points} says that if $A \subset \P^5 \setminus C$ is a finite set such that $|A| \leq 822$ and $\Gamma \cup A \subset \P^5$ is in linearly general position, then $\reg (\Gamma \cup A ) = \reg (\Gamma) =2001$.
\end{example}
 
Finally, we provide a cohomological characterization of finite schemes in uniform position of maximal regularity. More precisely, let $\Gamma \subset \P^n$ be a finite scheme in linearly general position of degree $d \geq 4n^2 +6n+1$ such that $\reg (\Gamma) =\left\lceil\frac{d-1}{n}\right\rceil+1$. If we write 
\begin{equation*}
d = nq+r+2 \quad \mbox{for some} \quad 0 \leq r \leq n-1,  
\end{equation*} 
then $\reg (\Gamma) =q+2$ and hence $h^1 (\P^n , \mathcal{I}_{\Gamma} (q)) > 0$. Furthermore, if $\Gamma$ lies on a rational normal curve, then one can easily show that
$$h^1 (\P^n , \mathcal{I}_{\Gamma} (q)) = r+1.$$
The following theorem shows that the converse is also true.

\begin{theorem}\label{thm:coh char max reg}
Let ${\Gamma} \subset \P^n$, $n \geq 2$, be a finite subscheme of degree $d \geq 4n^2+6n+1$ in linearly general position.
Write $d = nq+r+2$ for some $0 \leq r \leq n-1$.
Then
$$h^1 (\P^n , \mathcal{I}_{\Gamma} (q)) \leq r+1.$$
Moreover, the following two statements are equivalent:
\begin{enumerate}
\item[$(i)$] $\Gamma$ lies on a rational normal curve.
\item[$(ii)$] $h^1 (\P^n , \mathcal{I}_{\Gamma} (q)) = r+1$.
\end{enumerate}
In particular, if $\Gamma$ is a finite set such that $h^1 (\P^n , \mathcal{I}_{\Gamma} (q))>0$, then the following two statements are equivalent:
\begin{enumerate}
\item[$(i)$] $\Gamma$ is in uniform position.
\item[$(ii)$] $h^1 (\P^n , \mathcal{I}_{\Gamma} (q)) = r+1$.
\end{enumerate}
\end{theorem}

\begin{proof} 
Note that $h^1 (\P^n ,  \mathcal{I}_{\Gamma} (q))=0$ if and only if ${\rm reg}({\Gamma})<\left\lceil\frac{d-1}{n}\right\rceil+1$. So if $h^1 (\P^n ,  \mathcal{I}_{\Gamma} (q))=0$, then we are done. Now, we assume that $h^1 (\P^n ,  \mathcal{I}_{\Gamma} (q))>0$, or equivalently, that ${\rm reg}({\Gamma}) = \left\lceil\frac{d-1}{n}\right\rceil+1$. Let $C \subset \P^n$ be the rational normal curve of degree $n$ such that $\rho (\Gamma) = | \Gamma_0 |$ where $\Gamma_0 = \Gamma \cap C$ (cf. Theorem \ref{thm:main2}). Then $|\Gamma_0 | = nq+2+s$ for some $0 \leq s \leq r$ (cf. Corollary \ref{cor:maximal case}), and
we have the exact sequence
$$0 \rightarrow \mathcal{I}_C \rightarrow \mathcal{I}_{\Gamma_0} \rightarrow \mathcal{O}_{C} (-\Gamma_0 ) \rightarrow 0,$$
which enables us to verify that
$$I(C)_q = I(\Gamma_0 )_q \quad \mbox{and} \quad h^1 (\P^n ,  \mathcal{I}_{\Gamma_0} (q))=s+1.$$
Also, we have the exact sequence
$$0 \rightarrow \mathcal{I}_{\Gamma} \rightarrow \mathcal{I}_{\Gamma_0} \rightarrow \mathcal{O}_{\Gamma} (-\Gamma_0 ) \rightarrow 0$$
of coherent sheaves on $\P^n$. This induces the following cohomology long exact sequence:
$$H^0 (\P^n , \mathcal{I}_{\Gamma_0} (q)) \rightarrow H^0 (\Gamma , \mathcal{O}_{\Gamma} (-\Gamma_0 ) \otimes \mathcal{O}_{\P^n} (q)) \rightarrow H^1 (\P^n , \mathcal{I}_{\Gamma} (q)) \rightarrow H^1 (\P^n , \mathcal{I}_{\Gamma_0} (q)) \rightarrow 0$$
Then it follows that
$$h^1 (\P^n , \mathcal{I}_{\Gamma} (q)) \leq h^0 (\Gamma , \mathcal{O}_{\Gamma} (-\Gamma_0 ) \otimes \mathcal{O}_{\P^n} (q)) + h^1 (\P^n , \mathcal{I}_{\Gamma_0} (q)) = (r-s)+(s+1) = r+1.$$

Moreover, $h^1 (\P^n , \mathcal{I}_{\Gamma} (q)) = r+1$ if and only if the homomorphism
$$H^0 (\P^n , \mathcal{I}_{\Gamma_0} (q)) \rightarrow H^0 (\Gamma , \mathcal{O}_{\Gamma} (-\Gamma_0 ) \otimes \mathcal{O}_{\P^n} (q))$$
is the zero map. Since $\Gamma_0 = \Gamma \cap C$ and $I(C)_q = I(\Gamma_0 )_q$, this can happen exactly when $\Gamma = \Gamma_0$ and hence $\Gamma$ is contained in the rational normal curve $C$.

For the last statement, first assume that $\Gamma$ is in uniform position. Then $\Gamma$ lies on a rational normal curve of degree $n$ by Theorem \ref{thm:Nagel}. Thus, $h^1 (\P^n , \mathcal{I}_{\Gamma} (q)) = r+1$. Conversely, assume that $h^1 (\P^n , \mathcal{I}_{\Gamma} (q)) = r+1$. Then ${\rm reg}({\Gamma}) = \left\lceil\frac{d-1}{n}\right\rceil+1$ and hence $\Gamma$ lies on a rational normal curve of degree $n$ by Theorem \ref{thm:Nagel}.$(1)$. Obviously, any finite subset of a rational normal curve is in uniform position, which completes the proof.
\end{proof}

\end{document}